\documentclass[10pt,reqno]{amsart}

\newtheorem{thm}{Theorem}
\newtheorem{lemma}{Lemma}

\newtheorem{cor}{Corollary}
\newtheorem{prop}{Proposition}

\usepackage{amsmath}
\usepackage[psamsfonts]{amssymb}
\usepackage[mathscr]{euscript}

\newcommand{\cal}{\mathcal}

\title[Twistorial examples of Riemannian almost product manifolds
\hfill]{Twistorial examples of Riemannian almost product manifolds and their Gil-Medrano and Naveira types}

\author{Johann Davidov}
\thanks{The  author is partially supported by  the National Science
Fund, Ministry of Education and Science of Bulgaria under contract
DN 12/2 }

\address{Johann Davidov\\Institute of Mathematics and Informatics \\
Bulgarian Academy of Sciences\\ Acad. G.Bonchev Str. Bl.8\\ 1113
Sofia\\ Bulgaria} {\email{jtd@math.bas.bg}

\begin{document}

\begin{abstract}

Non-trivial examples of Riemannian almost product structures are
constructed on the product bundle of the positive and negative
twistor spaces of an oriented Riemannian four-manifold.  The
Gil-Medrano and Naveira types of these structures are determined and
a geometric interpretation of the corresponding classes is given.

\vspace{0,1cm} \noindent 2010 {\it Mathematics Subject
Classification}. Primary 53C15, Secondary 53C28

\vspace{0,1cm} \noindent {\it Key words: Riemannian almost product
manifolds, twistor spaces }
\end{abstract}

\thispagestyle{empty}

\maketitle \vspace{0.5cm}

\section{Introduction}

Recall that a Riemannian almost product manifold is a Riemannian
manifold $(N,h)$ endowed with a pair of orthogonal distributions $V$
and $H$ on $N$ such that $TN=V\oplus H$, $rank\, V< dim\,N$. A
Riemannian manifold of dimension $n$ admits an almost product
structure $(V,H)$ with $rank\,V=d$ if and only if the structure
group of the manifold can be reduced to the group $O(d)\times
O(n-d)$. The decomposition $TN=V\oplus H$ determines an orthogonal
isomorphism $P$ of the tangent bundle $TN$ with $P|V=Id$, $P|H=-Id$,
hence $P^2=Id$, $P\neq\pm Id$ at every point of $N$. Conversely, an
orthogonal isomorphism $P$ of $TN$ with $P_x^2=Id$, $P_x\neq\pm Id$
for every $x\in N$, defines an almost product structure on $(N,h)$
provided the dimension $d(x)$ of the $(+1)$-eigenspaces $V_x$ of
$P_x$ is constant. An isomorphism $P$ with these properties is also
called an almost product structure on $(N,h)$.  The distribution $V$
on which $P$ is the identity map is usually called vertical, while
the orthogonal distribution $H$ is called horizontal.

Similar to the Gray-Hervella classification  of almost Hermitian
manifolds \cite{GH}, A.M. Naveira \cite{Nav} has introduced 36
classes of Riemannian almost product manifolds. These come from an
orthogonal invariant decomposition under the action of the group
$O(d)\times O(n-d)$ on the space  of covariant $3$-tensor on an
Euclidean vector space having the same symmetries as the  the
covariant derivative of the fundamental form $\Phi(X,Y)=h(PX,Y)$ of
a Riemannian almost product manifold. This decomposition have been
found by Naveira [ibid.] and it has been proved by F.J. Carreras
\cite{C} that it is irreducible.

Naveira [ibid.], Gil-Medrano \cite{Med} and A. Montesinos \cite{Mon}
have given geometric interpretations of the Naveira classes. V.
Miquel \cite{Miq} has constructed examples for each class.

Gil-Medrano [ibid] has introduced algebraic conditions for the
covariant derivative of $P$ restricted to the distributions $V$ and
$H$ (see also Sec.~\ref{Gil-Medrano}) and has given their geometric
characterization. Combining one of these conditions on $V$ with one
on $H$, we can cover the 36 classes of Naveira.

A trivial example of a Riemannian almost product manifold is the
product $N=M_1\times M_2$ of Riemannian manifolds with $V=TM_1$ and
$H=TM_2$. In this paper, we use twistor theory to provide
non-trivial examples of Riemannina almost product manifolds.  Let
$(M,g)$ be an oriented four-dimensional Riemannian manifold, and let
${\mathcal Z}_{\pm}$ be the twistor spaces of $(M,g)$, the bundles
over $M$ whose sections are almost complex structures on $M$
compatible with the metric and $\pm$ the orientation. These are
$S^2$-bundles over $M$. The product bundle ${\mathscr P}={\mathcal
Z}_{+}\times {\mathcal Z}_{-}$ admits a natural $2$-parameter family
$G_{t_1,t_2}$, $t_1,t_2>0$, of Riemannian metrics and four
compatible almost product structures ${\mathcal K}_{\nu}$. We show
that these structures are not integrable, so they are not trivial
products even locally. We also find the Gil-Medrano types of
$({\mathscr P},G_{t_1,t_2},{\mathcal K}_{\nu})$, $\nu=1,...,4$,  in
terms of the curvature of the base manifold $(M,g)$ and specific
values of the parameters $t_1, t_2$. Using this, we determine the
Naveira classes of $({\mathscr P},G_{t_1,t_2},{\mathcal K}_{\nu})$.
Finally we give a geometric interpretation of the obtained results.

\section{Preliminaries}

\subsection{The twistor space of a four-manifold}

Let $(M,g)$ be an oriented  Riemannian manifold of dimension four.
The metric $g$ induces a metric on the bundle of two-vectors
$\pi:\Lambda^2TM\to M$ by the formula
$$
g(v_1\wedge v_2,v_3\wedge v_4)=\frac{1}{2}det[g(v_i,v_j)].
$$
The Levi-Civita connection of $(M,g)$ determines a connection on the
bundle $\Lambda^2TM$, both denoted by $\nabla$, and the
corresponding curvatures are related by
$$
R(X\wedge Y)(Z\wedge T)=R(X,Y)Z\wedge T+Z\wedge R(X,Y)T
$$
for $X,Y,Z,T\in TM$. The  curvature operator ${\cal R}$ is the
self-adjoint endomorphism of $\Lambda ^{2}TM$ defined by
$$
 g({\cal R}(X\land Y),Z\land T)=g(R(X,Y)Z,T)
$$
Let us note that we adopt the following definition for the curvature
tensor $R$ : $R(X,Y)=\nabla_{[X,Y]}-[\nabla_{X},\nabla_{Y}]$.

 The Hodge star operator defines an endomorphism $\ast$ of
$\Lambda^2TM$ with $\ast^2=Id$. Hence we have the orthogonal
decomposition
$$
\Lambda^2TM=\Lambda^2_{-}TM\oplus\Lambda^2_{+}TM
$$
where $\Lambda^2_{\pm}TM$ are the subbundles of $\Lambda^2TM$
corresponding to the $(\pm 1)$-eigenvalues of the operator $\ast$.

\smallskip

Let $(E_1,E_2,E_3,E_4)$ be a local oriented orthonormal frame of
$TM$. Set
\begin{equation}\label{s-basis}
s_1^{\pm}=E_1\wedge E_2 \pm E_3\wedge E_4, \quad s_2^{\pm}=E_1\wedge
E_3\pm E_4\wedge E_2, \quad s_3^{\pm}=E_1\wedge E_4\pm E_2\wedge
E_3.
\end{equation}
Then $(s_1^{\pm},s_2^{\pm},s_3^{\pm})$ is a local orthonormal frame
of $\Lambda^2_{\pm}TM$ defining an orientation on
$\Lambda^2_{\pm}TM$, which does not depend on the choice of the
frame $(E_1,E_2,E_3,E_4)$  (see, for example, \cite{D15}).

\smallskip

For every $a\in\Lambda ^2TM$, define a skew-symmetric endomorphism
$K_a$ of $T_{\pi(a)}M$ by

\begin{equation}\label{cs}
g(K_{a}X,Y)=2g(a, X\wedge Y), \quad X,Y\in T_{\pi(a)}M.
\end{equation}

It is easy to check that:
\begin{equation}\label{com-anticom}
\begin{array}{c}
K_a\circ K_b=K_b\circ K_a,\quad {\rm{if}}\quad a\in\Lambda^2_{\pm}T_pM,\> b\in\Lambda^2_{\mp}T_pM;\\[6pt]
K_a\circ K_b=-K_b\circ K_a,\quad {\rm{if}}\quad
a,b\in\Lambda^2_{\pm}T_pM, \> a\perp b.
\end{array}
\end{equation}

Note also that, denoting by $\gamma$ the standard metric
$-\frac{1}{2}Trace\,PQ$ on the space of skew-symmetric
endomorphisms, we have $\gamma(K_a,K_b)=2g(a,b)$ for $a,b\in \Lambda
^2TM$. If $\sigma\in\Lambda^2_{\pm}TM$ is a unit vector, then
$K_{\sigma}$ is a complex structure on the vector space
$T_{\pi(\sigma)}M$ compatible with the metric and $\pm$ the
orientation of $M$. Conversely, the $2$-vector $\sigma$ dual to one
half of the fundamental $2$-form of such a complex structure is a
unit vector in $\Lambda^2_{\pm}TM$. Thus, the unit sphere subbunlde
${\cal Z}_{\pm}$ of $\Lambda^2_{\pm}TM$ parametrizes the complex
structures on the tangent spaces of $M$ compatible with its metric
and $\pm$ orientation. The subbundles ${\mathcal Z}_{+}$ and
${\mathcal Z}_{-}$ are called the postive and the negative twistor
space of $M$. They are the two connected components of the bundle
over $M$ whose fibre at a point $p\in M$ consists of all complex
structures on $T_pM$ compatible with the metric.

\smallskip

The connection $\nabla$ on $\Lambda^2TM$ induced by the Levi-Civita
connection of $M$ preserves the bundles $\Lambda^2_{\pm}TM$, so it
induces a metric connection on each of them denoted again by
$\nabla$. The horizontal distribution of $\Lambda^2_{\pm}TM$ with
respect to $\nabla$ is tangent to the twistor space ${\mathcal
Z}_{\pm}$. Thus, we have the decomposition $T{\mathcal
Z}_{\pm}={\mathcal H}\oplus {\mathcal V}$ of the tangent bundle of
${\mathcal Z}_{\pm}$ into horizontal and vertical components. The
vertical space ${\mathcal V}_{\tau}=\{V\in T_{\tau}{\mathcal
Z}_{\pm}:~ \pi_{\ast}V=0\}$ at a point $\tau\in{\mathcal Z}_{\pm}$
is the tangent space to the fibre of ${\mathcal Z}_{\pm}$ through
$\tau$. Considering $T_{\tau}{\mathcal Z}_{\pm}$ as a subspace of
$T_{\tau}(\Lambda^2_{\pm}TM)$, ${\mathcal V}_{\tau}$ is the
orthogonal complement of $\tau$ in $\Lambda^2_{\pm}T_{\pi(\tau)}M$.
The map $V\ni{\cal V}_{\tau}\to K_{V}$ gives an identification of
the vertical space with the space of skew-symmetric endomorphisms of
$T_{\pi(\tau)}M$ which anti-commute with $K_{\tau}$. Let $s$ be a
local section of ${\mathcal Z}_{\pm}$ such that $s(p)=\tau$ where
$p=\pi(\tau)$. Considering $s$ as a section of $\Lambda^2_{\pm}TM$,
we have $\nabla_{X}s\in{\mathcal V}_{\tau}$ for every $X\in T_pM$
since $s$ has a constant length. Moreover,
$X^h_{\tau}=s_{\ast}X-\nabla_{X}s$ is the horizontal lift of $X$ at
${\tau}$.

\smallskip

Denote by $\times$ the usual vector cross product on the oriented
$3$-dimensional vector space $\Lambda^2_{\pm}T_pM$, $p\in M$,
endowed with the metric $g$. Then it is easy to check that

\begin{equation}\label{eq 2.5}
g(R(a)b,c)=\pm g({\cal R}(a),b\times c))
\end{equation}
for $a\in\Lambda^2T_{p}M$, $b,c\in \Lambda ^{2}_{\pm}T_{p}M$. Also
\begin{equation}\label{KK}
K_{b}\circ K_{c}=-g(b,c)Id\pm K_{b\times c},\quad b,c\in \Lambda
^{2}_{\pm}T_{p}M.
\end{equation}

\smallskip

Denote by ${\cal B}:\Lambda^2TM\to \Lambda^2TM$ the endomorphism
corresponding to the traceless Ricci tensor. If  $s$ denotes the
scalar curvature of $(M,g)$ and $\rho:TM\to TM$ is the Ricci
operator, $g(\rho(X),Y)=Ricci(X,Y)$, we have
$$
{\mathcal B}(X\wedge Y)=\rho(X)\wedge
Y+X\wedge\rho(Y)-\frac{s}{2}X\wedge Y.
$$
Note that ${\mathcal B}$ sends  $\Lambda^2_{\pm}TM$ into
$\Lambda^2_{\mp}TM$. Let ${\cal W}: \Lambda^2TM\to \Lambda^2TM$ be
the endomorphism corresponding to the Weyl conformal tensor. Denote
the restriction of ${\cal W}$ to $\Lambda^2_{\pm}TM$ by ${\cal
W}_{\pm}$, so ${\cal W}_{\pm}$ sends $\Lambda^2_{\pm}TM$ to
$\Lambda^2_{\pm}TM$ and vanishes on $\Lambda^2_{\mp}TM$. Moreover,
$Trace\,{\mathcal W}_{\pm}=0$.

 It is well known that the curvature operator decomposes as (\cite{ST}, see
e.g. \cite[Chapter 1 H]{Besse})
\begin{equation}\label{dec}
{\cal R}=\frac{s}{6}Id+{\mathcal B}+{\mathcal W}_{+}+{\mathcal
W}_{-}
\end{equation}
Note that this differs by  a factor $1/2$ from \cite{Besse} because
of the factor $1/2$ in our definition of the induced metric on
$\Lambda^2TM$. Note also that changing the orientation of $M$
interchanges the roles of $\Lambda^2_{+}TM$ and $\Lambda^2_{-}TM$,
correspondingly the roles of ${\mathcal W}_{+}$ and ${\mathcal
W}_{-}$.

\smallskip

The Riemannian manifold $(M,g)$ is Einstein exactly when ${\cal
B}=0$. It is called anti-self-dual (self-dual), if ${\mathcal
W}_{+}=0$ (resp. ${\mathcal W}_{-}=0$). By a famous result of
Atiyah-Hitchin-Singer \cite{AHS}, the anti-self-duality
(self-duality) condition is necessary and sufficient for
integrability of a naturally defined almost complex structure on
${\mathcal Z}_{+}$ (resp., ${\mathcal Z}_{-}$).

\section{Riemannian almost product structure on the product bundle ${\mathcal Z}_{+}\times {\mathcal Z}_{-}$}

Let ${\mathscr P}={\mathcal Z}_{+}\times{\mathcal Z}_{-}$ be the
product bundle over $M$ of the bundles ${\mathcal Z}_{\pm}$.

\smallskip

The projection to $M$ of the vector bundle $\Lambda^2_{+}TM\oplus
\Lambda^2_{-}TM$ will be denoted by $\pi$ and we shall use the same
symbol for its restriction to the subbundle ${\mathscr P}={\mathcal
Z}_{+}\times{\mathcal Z}_{-}$. By abuse of notation, the direct sum
of the connections $\nabla|\Lambda^2_{\pm}$ will also be denoted by
$\nabla$.

\smallskip

Let $\varkappa=(\sigma^{+},\sigma^{-})\in {\mathscr P}$ and
$p=\pi(\varkappa)$. Take sections $\phi^{\pm}$ of $\Lambda^2_{\pm}$
such that $\phi^{\pm}(p)=\sigma^{\pm}$ and $\nabla\phi^{\pm}|_p=0$.
Then
$\Phi=(\frac{\phi^{+}}{||\phi^{+}||},\frac{\phi^{-}}{||\phi^{-}||})$
is a section of $\Lambda^2_{+}TM\oplus\Lambda^2_{-}TM$ taking values
in ${\mathscr P}$ and such that $\Phi(p)=\varkappa$,
$\nabla\Phi|_p=0$. Hence the horizontal space ${\mathcal
H}_{\varkappa}=\Phi_{\ast}(T_pM)$ of connection $\nabla$ on
$\Lambda^2_{+}TM\oplus\Lambda^2_{-}TM$ at $\varkappa$ is tangent to
the submanifold ${\mathscr P}$. Thus, we have the decomposition
$T_{\varkappa}{\mathscr P}={\mathcal H}_{\varkappa}\oplus {\mathcal
V}_{\varkappa}$ into horizontal and vertical parts, where the
vertical space ${\mathcal V}_{\varkappa}$ of the bundle ${\mathscr
P}\to M$ is clearly the product of the vertical spaces ${\mathcal
V}_{\sigma^{\pm}}$ of the bundles ${\mathcal Z}_{\pm}\to M$. This
decomposition allows one to define four almost product structures
${\mathcal K}_{\nu}$ on ${\mathscr P}$, $\nu=1,..,4$, setting
$$
\begin{array}{c}
{\mathcal K}_{\nu}X^h_{\varkappa}=(K_{\sigma^{+}}\circ K_{\sigma^{-}}X)^h_{\varkappa}\quad {\rm{for}}\quad X\in T_pM,\\[6pt]

{\mathcal K}_{1}(V^{+},V^{-})=(V^{+},V^{-}),\quad {\mathcal K}_{2}(V^{+},V^{-})=(V^{+},-V^{-}),\\[8pt]
{\mathcal K}_{3}(V^{+},V^{-})=-(V^{+},V^{-}),\quad {\mathcal K}_{4}(V^{+},V^{-})=(-V^{+},V^{-})\\[8pt]
{\rm{for}}\quad (V^{+},V^{-})\in{\mathcal V}_{\varkappa}.
\end{array}
$$
It is convenient to set $\varepsilon_1=\varepsilon_2=1$,
$\varepsilon_3=\varepsilon_4=-1$, so that
$$
{\mathcal
K}_{\nu}(V)=\varepsilon_{\nu}(V^{+},(-1)^{\nu+1}V^{-}),\quad
V=(V^{+},V^{-})\in{\mathcal V}_{\varkappa},\quad \nu=1,...,4.
$$

 Clearly ${\mathcal K}_{\nu}^2=Id$. The endomorphism $P_{\varkappa}=K_{\sigma^{+}}\circ K_{\sigma^{-}}$ of $T_pM$ is
an involution different form $\pm Id$ and its $\pm 1$ eigenspaces
are invariant under $K_{\sigma^{+}}$ and $K_{\sigma^{-}}$. Hence  we
can find an oriented orthonormal basis $E_1,...,E_4$ of $T_pM$ such
that $K_{\sigma^{+}}E_1=K_{\sigma^{-}}E_1=E_2$ and
$K_{\sigma^{+}}E_3=-K_{\sigma^{-}}E_3=E_4$. Then
$P_{\varkappa}E_i=-E_i$ for $i=1,2$ and $P_{\varkappa}E_j=E_j$ for
$j=3,4$. Therefore the dimensions of the $(+1)$  and
$(-1)$-eigenspaces of ${\mathcal K}_{1}, ..., {\mathcal K}_{4}$ are
$(6,2)$, $(4,4)$, $(2,6)$, $(4,4)$, respectively. Thus, ${\mathcal
K}_{\nu}$, $\nu=1,...,4$, are almost product structures on the
manifold ${\mathscr P}$.

\medskip

 For ${\bf t}=(t_1,t_2)$ with $t_1>0$, $t_2>0$, define a
$2$-parameter family of Riemannian metrics $G_{\bf t}$ on ${\mathscr
P}$ by
$$
G_{\bf t}(X^h+V,
Y^h+W)_{\varkappa}=g(X,Y)+t_1g(V^{+},W^{+})+t_2g(V^{-},W^{-}),
$$
where $X,Y\in T_{\pi(\varkappa)}M$ and $V=(V^{+},V^{-}),
W=(W^{+},W^{-})\in {\mathcal V}_{\varkappa}$.

Clearly, the projection $\pi: ({\mathscr P}, G_{\bf t})\to (M,g)$ is
a Riemannian submersion. Moreover,  the almost product structures
${\mathcal K}_{\nu}$ are compatible with every metric $G_{\bf t}$.


\smallskip
Let $({\mathscr U},x_1,...,x_4)$ be a local coordinate system of
$M$, and let $(E_1,...,E_4)$ be an oriented orthonormal frame of
$TM$ on ${\mathscr U}$. If $s_i^{\pm}$, $i=1,2,3$, are the local
frames of $\Lambda^2_{\pm}TM$ define by (\ref{s-basis}), for
$\varkappa=(\sigma^{+},\sigma^{-})\in \pi^{-1}({\mathscr U})$, set
$\widetilde x_{\alpha}=x_{\alpha}\circ\pi$,
$y^{\pm}_j(\varkappa)=g(\sigma^{\pm},
(s_j^{\pm}\circ\pi)(\varkappa))$, $1\leq \alpha \leq 4$, $1\leq
j\leq 3$. Then $\{\widetilde x_{\alpha},y^{+}_j,y^{-}_j\}$ are local
coordinates of the manifold $\Lambda^2_{+}TM\oplus \Lambda^2_{-}TM$
on $\pi^{-1}({\mathscr U})$.

   The horizontal lift $X^h$ on $\pi^{-1}({\mathscr U})$ of a vector field
$$
X=\sum_{\alpha=1}^4 X^{\alpha}\frac{\partial}{\partial x_{\alpha}}
$$
is given by
\begin{equation}\label{hl}
\begin{array}{c}
X^h=\displaystyle{\sum\limits_{\alpha=1}^4 (X^{\alpha}\circ\pi)\frac{\partial}{\partial \widetilde{x}_{\alpha}}}\\[8pt]
-\displaystyle{\sum\limits_{j,k=1}^3y^{+}_j(g(\nabla_{X}s^{+}_j,s^{+}_k)\circ\pi)\frac{\partial}{\partial
y^{+}_k}
-\sum\limits_{j,k=1}^3y^{-}_j(g(\nabla_{X}s^{-}_j,s^{-}_k)\circ\pi)\frac{\partial}{\partial
y^{-}_k}}.
\end{array}
\end{equation}
Hence
\begin{equation}\label{Lie-1}
\begin{array}{c}
[X^h,Y^h]=[X,Y]^h\\[8pt]
+\displaystyle{}\sum\limits_{j,k=1}^3y_j^{+}(g(R(X\wedge
Y)s_j^{+},s_k^{+})\circ\pi)\frac{\partial}{\partial y_k^{+}}
+\sum\limits_{j,k=1}^3y_j^{-}(g(R(X\wedge
Y)s_j^{-},s_k^{-})\circ\pi)\frac{\partial}{\partial y_k^{-}}.
\end{array}
\end{equation}
for every vector fields $X,Y$ on ${\mathscr U}$. Using the standard
identification
$$
T_{\omega}(\Lambda^2_{\pm}T_{\pi(\omega)}M)\cong
\Lambda^2_{\pm}T_{\pi(\omega)}M, \quad \omega\in{\mathcal Z}_{\pm},
$$
we obtain from (\ref{Lie-1}) the well-known formula
\begin{equation}\label{Lie-2}
[X^h,Y^h]_{\varkappa}=[X,Y]^h_{\varkappa}+R_{p}(X\wedge Y)\varkappa,
\quad \varkappa=(\sigma^{+},\sigma^{-})\in{\mathcal
Z}_{+}\times{\mathcal Z}_{-}, \quad p=\pi(\varkappa),
\end{equation}
where $R_{p}(X\wedge Y)\varkappa=(R_{p}(X\wedge Y)\sigma^{+},
R_{p}(X\wedge Y)\sigma^{-})\in{\mathcal V}_{\varkappa}={\mathcal
V}_{\sigma^{+}}\times{\mathcal V}_{\sigma^{-}}$.

\smallskip

Note also that it follows from (\ref{eq 2.5}) that if
$\varkappa=(\sigma^{+},\sigma^{-})\in{\mathcal Z}_{+}\times{\mathcal
Z}_{-}$ and $V=(V^{+},V^{-})\in{\mathcal V}_{\varkappa}$, $X,Y\in
T_{\pi(\varkappa)}M$,
\begin{equation}\label{RJ,V}
G_{\bf t}(R(X,Y)\varkappa,V)=g({\mathcal R}(t_1\sigma^{+}\times
V^{+} - t_2\sigma^{-}\times V^{-}),X\wedge Y).
\end{equation}

 For any (local) section $a=(a^{+},a^{-})$ of $\Lambda^2_{+}TM\oplus \Lambda^2_{-}TM$, 
denote by $\widetilde a=(\widetilde a^{+},\widetilde a^{-})$ the
vertical vector field on ${\mathscr P}$ defined by
\begin{equation}\label{tilde a}
\widetilde
a^{\pm}_{\varkappa}=a^{\pm}(p)-g(a^{\pm}(p),\sigma^{\pm})\sigma^{\pm},
\quad \varkappa=(\sigma^{+},\sigma^{-}), \quad p=\pi(\varkappa).
\end{equation}

\smallskip

Note that for every $\varkappa\in {\mathscr P}$ we can find sections
$a_1,...,a_4$ of $\Lambda^2_{+}TM\oplus \Lambda^2_{-}TM$ near the
point $p=\pi(\varkappa)$ such that $\widetilde a_1,...,\widetilde
a_4$ form a basis of the vertical vector space at each point in a
neighbourhood of $\varkappa$.

\vspace{0.1cm}

The next lemma is a kind of folklore appearing in different contexts
(cf, for example, \cite{G,D05}).

\begin{lemma}\label{Lie-hor-ver}
Let $X$ be a vector field on $M$ and let $a=(a^{+},a^{-})$ be a
section of $\Lambda^2_{+}TM\oplus \Lambda^2_{-}TM$  defined on a
neighbourhood of the point $p=\pi(\varkappa)$,
$\varkappa\in{\mathscr P}$. Then:
$$
[X^h,\widetilde{a}]_{\varkappa}=\widetilde{(\nabla_X a)}_{\varkappa}
$$
\end{lemma}

\begin{proof}  Fix a point $p\in M$, take  an oriented orthonormal
frame $(E_1,...,E_4)$ of $TM$ such that $\nabla E_i|_p=0$, and
define $s_i^{\pm}$, $i=1,2,3$,  by (\ref{s-basis}). Set
$a^{\pm}=\sum\limits_{i=1}^3 a_i^{\pm}s_i^{\pm}$. Then, in the local
coordinates of $\Lambda^2_{+}TM\oplus \Lambda^2_{-}TM$ introduced
above,
$$
\widetilde a=\sum_{i=1}^3[\widetilde
a_{i}^{+}\frac{\partial}{\partial y_{i}^{+}}+\widetilde
a_{i}^{-}\frac{\partial}{\partial y_{i}^{-}}],
$$
where
$$
\widetilde
a^{\pm}=\sum\limits_{i=1}^3[a_i^{\pm}\circ\pi-y_i^{\pm}\sum\limits_{j=1}^3y_j^{\pm}(a_j^{\pm}\circ\pi)]\frac{\partial}{\partial
y_i^{\pm}}
$$
Let us also note that for every vector field $X$ on $M$ near the
point $p$, we have by (\ref{hl})
$$
X_{\varkappa}^h=\displaystyle{\sum_{\alpha=1}^{4}
X^{\alpha}(p)\frac{\partial}{\partial\tilde
x_{\alpha}}(\varkappa)},\quad
\displaystyle{[X^h,\frac{\partial}{\partial
y_{i}^{\pm}}}]_{\varkappa}=0,\quad i=1,2,3,
$$
since $\nabla s_{i}^{\pm}|_p=0$, $i=1,2,3$. Hence
$$
\begin{array}{r}
[X^h,\widetilde
a]_{\varkappa}=\sum\limits_{i=1}^3[X(a^{+}_i)-y^{+}_i(\varkappa)\sum\limits_{j=1}^3y^{+}_j(\varkappa)X_p(a^{+}_j)]\displaystyle{\frac{\partial}{\partial
 y^{+}_{j}}(\varkappa)}\\[10pt]
+\sum\limits_{i=1}^3[X(a^{-}_i)-y^{-}_i(\varkappa)\sum\limits_{j=1}^3y^{-}_j(\varkappa)X_p(a^{-}_j)]\displaystyle{\frac{\partial}{\partial
 y^{-}_{j}}(\varkappa)}.
\end{array}
$$
On the other hand, considering $\widetilde a$ as a section of
$\Lambda^2_{+}TM\oplus \Lambda^2_{-}TM$, we have for $X\in T_pM$
$$
\begin{array}{r}
D_{X}\widetilde
a=\sum\limits_{i=1}^3[X(a^{+}_i)-y^{+}_i(\varkappa)\sum\limits_{j=1}^3y^{+}_j(\varkappa)X_p(a^{+}_j)]s_i^{+}(p)\\[8pt]
+\sum\limits_{i=1}^3[X(a^{-}_i)-y^{-}_i(\varkappa)\sum\limits_{j=1}^3y^{-}_j(\varkappa)X_p(a^{-}_j)]s_i^{-}(p).
\end{array}
$$

This proves the lemma.
\end{proof}

Denote by $D$ the Levi-Civita connection of $({\mathcal P},G_{\bf
t})$.

\smallskip

Let $\varkappa=(\sigma^{+},\sigma^{-})\in{\mathscr P}$ and
$p=\pi(\varkappa)$. As we have noticed, we can find an oriented
orthonormal basis $(E_1,...,E_4)$ of $T_pM$ such that
$\sigma^{\pm}=E_1\wedge E_2 \pm E_3\wedge E_4$. Extend this basis to
an oriented orthonormal frame of vector fields in a neighbourhood of
$p$ such that $\nabla E_{\alpha}|_p=0$, ${\alpha}=1,...,4$. Define
$s^{\pm}_i$, $i=1,2,3$,  by (\ref{s-basis}), so that
$s_1^{\pm}(p)=\sigma^{\pm}$ and $\nabla s^{\pm}_i|_p=0$. The
vertical vector fields  $\widetilde{a}_1,...,\widetilde{a}_4$
determined  by the sections $a_1=(s_2^{+},0)$, $a_2=(s_3^{+},0)$,
$a_3=(0,s_2^{-})$, $a_4=(0,s_3^{-})$ of $\Lambda^2_{+}TM\oplus
\Lambda^2_{-}TM$ form a frame of the vertical bundle ${\mathcal V}$
of ${\mathscr P}$ in a neighbourhood of $\varkappa$. Let $V\in
{\mathcal V}_{\varkappa}$ and let $v$ be a section of
$\Lambda^2_{+}TM\oplus \Lambda^2_{-}TM$ such that $v(p)=V$ and
$\nabla v|_p=0$. Denote by $\widetilde v$ the vertical vector field
corresponding to this section. By Lemma~\ref{Lie-hor-ver},
$[X^h,\widetilde{a}_{l}]_{\varkappa}=
[X^h,\widetilde{v}]_{\varkappa}=0$, $l=1,...,4$,  for every vector
field $X$ in a neighbourhood of $p$.  It follows from the Koszul
formula for the Levi-Civita connection that the vectors
$(D_{\widetilde{v}}\widetilde{a}_{l})_{\varkappa}$ for all
$l=1,...,4$ are $G_{\bf t}$-orthogonal to every horizontal vector
$X^h_{\varkappa}$. Hence $D_{V}\widetilde{a}_{l}$ are vertical
tangent vectors of ${\mathscr P}$ at $\varkappa$. It follows that,
for every vertical vector field $W$, $D_{V}W$ is a vertical vector
field. Thus, the fibres of ${\mathscr P}$ are totally geodesic
submanifolds. This, of course, follows also from the Vilms theorem
(see, for example, \cite[Theorem 9.59]{Besse}.

The proof of the following lemma is practically given  in
\cite{DM91,D05}) and we present it here just for completeness.

\begin{lemma}\label{LC} If $X,Y$ are vector
fields on $M$ and $V=(V^{+},V^{-})$ is a vertical vector field on
${\mathscr P}$, then
\begin{equation}\label{D-hh}
(D_{X^h}Y^h)_{\varkappa}=(\nabla_{X}Y)^h_{\varkappa}+\frac{1}{2}R(X,Y)\varkappa.
\end{equation}
\begin{equation}\label{D-vh}
(D_{V}X^h)_{\varkappa}={\mathcal
H}(D_{X^h}V)_{\varkappa}=-\frac{1}{2}(R_{p}(t_1\sigma^{+}\times
V^{+}-t_2\sigma^{-}\times V^{-})X)_{\varkappa}^h
\end{equation}
where $\varkappa=(\sigma^{+},\sigma^{-})\in {\mathscr P}$,
$p=\pi(\varkappa)$, and ${\cal H}$ means "the horizontal component".
\end{lemma}

\begin{proof}
The Koszul formula, identity (\ref{Lie-1}), and
Lemma~\ref{Lie-hor-ver} imply
$$
(D_{X^h}Y^h)_{\varkappa}=(\nabla_{X}Y)^h_{\varkappa}+\frac{1}{2}R(X,Y)\varkappa.
$$
Next, $D_{V}X^h$ is orthogonal to any vertical vector field $W$
since $D_{V}W$ is a vertical vector field. Thus $D_{V}X^h$ is a
horizontal vector field. Hence $D_{V}X^h={\cal H}D_{X^h}V$ since
$[V,X^h]$ is a vertical vector field. Therefore
\begin{equation}\label{v-hh}
\begin{array}{c}
G_{\bf t}(D_{V}X^h,Y^h)_{\varkappa}=G_{\bf t}(D_{X^h}V,Y^h)_{\varkappa}=-G_{\bf t}(V,D_{X^h}Y^h)_{\varkappa}\\[6pt]
=-\displaystyle{\frac{1}{2}} G_{\bf t}(R(X,Y)\varkappa,V)
\end{array}
\end{equation}
Thus, (\ref{D-vh})  follows from (\ref{RJ,V}).
\end{proof}

Set
$$
F_{\bf t,\nu}(A,B)=G_{\bf t}({\mathcal K}_{\nu}A,B),\quad A,B\in
T{\mathscr P}.
$$

\begin{cor}\label{cor-DFhhv}
Let $\varkappa=(\sigma^{+},\sigma^{-})\in{\mathscr P}$, $X,Y\in
T_{\pi(\varkappa)}M$, $V\in{\mathcal V}_{\varkappa}$. Then
$$
(D_{X^h_{\varkappa}}F_{\bf t,\nu})(Y^h,U)=-\frac{1}{2}G_{\bf
t}({\mathcal K}_{\nu}R(X,Y)\varkappa,U) +\frac{1}{2}G_{\bf
t}(R(X,P_{\varkappa}Y)\varkappa,U),
$$
where $P_{\varkappa}=K_{\sigma^{+}}\circ K_{\sigma^{-}}$.
\end{cor}

\begin{proof}
This follows from  the identity
$$
(D_{X^h_{\varkappa}}F_{\bf t,\nu})(Y^h,U)=-G_{\bf t}({\mathcal
K}_{\nu}D_{X^h_{\varkappa}}Y^h,U)-G_{\bf
t}((P_{\varkappa}Y)^h_{\varkappa},D_{X^h_{\varkappa}}U)
$$
and identities (\ref{D-hh}), and (\ref{v-hh}).

\end{proof}

\begin{lemma}\label{D-F}
Let $\varkappa=(\sigma^{+},\sigma^{_-})\in{\mathscr P}$, $X,Y,Z\in
T_{\pi(\varkappa)}M$, and $U,V,W\in{\mathcal V}_{\varkappa}$. Then:

\noindent
\begin{itemize}
\item [(i)]\quad $(D_{X^h_{\varkappa}}F_{\bf t,\nu})(Y^h,Z^h)=0$;
\bigskip
\item [(ii)]
$
\begin{array}{c}
(D_{X^h_{\varkappa}}F_{\bf
t,\nu})(Y^h,U)=-\frac{1}{2}\varepsilon_{\nu}g({\mathcal
R}(t_1\sigma^{+}\times U^{+}+(-1)^{\nu}t_2\sigma^{-}\times
U^{-}),X\wedge
Y)\\[6pt]
+\frac{1}{2}g({\mathcal R}(t_1\sigma^{+}\times
U^{+}-t_2\sigma^{-}\times U^{-}),X\wedge
K_{\sigma^{+}}K_{\sigma^{-}}Y);
\end{array}
$
\bigskip

\item [(iii)]
\noindent $
\begin{array}{c}
\hspace{-2.5cm} (D_{U}F_{\bf t,\nu})(Y^h,Z^h)_{\varkappa}=g_p((K_{\sigma^{-}}K_{U^{+}}+K_{\sigma^{+}}K_{U^{-}})Y,Z)\\[6pt]
+\frac{1}{2}g({\mathcal R}(t_1\sigma^{+}\times
U^{+}-t_2\sigma^{-}\times U^{-}),Y\wedge
K_{\sigma^{+}}K_{\sigma^{-}}Z-K_{\sigma^{+}}K_{\sigma^{-}}Y\wedge
Z);
\end{array}
$
\bigskip
\item [(iv)]\quad $(D_{X^h_{\varkappa}}F_{\bf t,\nu})(U,V)=0$;

\bigskip
\item [(v)]\quad $(D_{U}F_{\bf t,\nu})(X^h,V)=0$;

\bigskip
\item [(vi)]\quad $(D_{U}F_{\bf t,\nu})(V,W)=0$;
\end{itemize}

\end{lemma}

\begin{proof}

Take an  oriented orthonormal basis $E_1,...,E_4$ of $T_pM$ such
that $\sigma^{\pm}=E_1\wedge E_2\pm E_3\wedge E_4$. Extend the basis
$E_1,...,E_4$ to an oriented orthormal frame in a neighbourhood of
the point $p$ such that $\nabla E_{\alpha}|_p=0$,
${\alpha}=1,...,4$. Using this frame, define sections $s_i^{\pm}$,
$i=1,2,3$, of $\Lambda^2_{\pm}TM$ by (\ref{s-basis}); clearly
$\nabla s^{\pm}_i|_p=0$. Also, extend $Y$ and $Z$ to vector fields
such that $\nabla Y|_p=\nabla Z|_p=0$. Then
$$
\begin{array}{c}
(D_{X^h_{\varkappa}}F_{\bf t,\nu})(Y^h,Z^h)=X^h_{\varkappa}(G_{\bf t}({\mathcal K}_{\nu}Y^h,Z^h)) \\[6pt]
-G_{\bf t}({\mathcal K}_{\nu}D_{X^h_{\varkappa}}Y^h,Z^h) -G_{\bf
t}(Y^h,{\mathcal K}_{\nu}D_{X^h_{\varkappa}}Z^h)
=X^h_{\varkappa}(G_{\bf t}({\mathcal K}_{\nu}Y^h,Z^h))
\end{array}
$$
since ${\mathcal K}_{\nu}D_{X^h_{\varkappa}}Y^h$ and ${\mathcal
K}_{\nu}D_{X^h_{\varkappa}}Z^h$ are vertical vectors by
(\ref{D-hh}). Setting $S=(s_1^{+},s_1^{-})$, we get a section of
${\mathscr P}$ with $S(p)=\varkappa$, $\nabla S|_p=0$. Hence
$$
\begin{array}{c}
X^h_{\varkappa}(G_{\bf t}({\mathcal K}_{\nu}Y^h,Z^h))=X_p(G_{\bf
t}({\mathcal K}_{\nu}Y^h,Z^h))\circ S)=X_p(g(K_{s^{+}_1}\circ
K_{s^{-}_1}Y,Z))\\[6pt]
=X_p(-\sum\limits_{k=1}^2g(E_k,Y)g(E_k,Z)+\sum\limits_{l=3}^4
g(E_l,Y)g(E_l,Z))=0
\end{array}
$$
since $\nabla E_{\alpha}|_p=\nabla Y|_p=\nabla Z|_p=0$. This proves
identity (i).

\smallskip

Extending the vector $U$ to a vertical vector field in a
neighbourhood of $\varkappa$, we see that
$$
(D_{X^h_{\varkappa}}F_{\bf t,\nu})(Y^h,U)=-G_{\bf
t}(D_{X^h}Y^h,{\mathcal K}_{\nu}U)_{\varkappa}-G_{\bf t}({\mathcal
K}_{\nu}Y^h,D_{U}X^h)_{\varkappa}
$$
since the vector ${\mathcal K}_{\nu}Y^h$ is horizontal, while U and
$[X^h,U]$ are vertical. Thus, the second formula of the lemma
follows from (\ref{D-hh}), (\ref{D-vh}), and (\ref{RJ,V}).

\smallskip

Formula (ii) follows from Corollary~\ref{cor-DFhhv} and
(\ref{RJ,V}).

\smallskip

Formula (iii) follows from (\ref{D-vh}) and the identity
$$
\begin{array}{c}
U(G_t({\mathcal K}_{\nu}Y^h,Z^h))=\sum\limits_{i,j=1}^3 U(y^{+}_i y^{-}_j (g(K_{s^{+}_i}K_{s^{-}_j}Y,Z)\circ\pi))\\[8pt]
=g_p((K_{U^{+}}K_{\sigma^{-}}+K_{\sigma^{+}}K_{U^{-}})Y,Z).
\end{array}
$$

\smallskip

To prove (iv), take sections $a=(a^{+},a^{-})$ and $b=(b^{+},b^{-})$
of $\Lambda^2_{+}TM\oplus\Lambda^2_{-}TM$ such that $a(p)=U$,
$b(p)=V$ and $\nabla a|_p=\nabla b|_p=0$. Let $\widetilde
a=(\widetilde a^{+},\widetilde a^{-})$ and $\widetilde b=(\widetilde
b^{+},\widetilde b^{-})$ be the vertical vector fields on ${\mathscr
P}$ defined by means of $a$ and $b$ via (\ref{tilde a}). Then
$\widetilde a(\varkappa)=U$, $\widetilde b(\varkappa)=V$, and
$[X^h,\widetilde a]_{\varkappa}=[X^h,\widetilde b]_{\varkappa}=0$ by
Lemma~\ref{Lie-hor-ver}. Hence $D_{X^h_{\varkappa}}\widetilde a$ and
$D_{X^h_{\varkappa}}\widetilde b$ are horizontal vectors by
(\ref{D-vh}). Thus,
$$
(D_{X^h_{\varkappa}}F_{\bf t})(U,V)=X^h_{\varkappa}(G_t({\mathcal
K}_{\nu}\widetilde a,\widetilde b)).
$$
We have $X^h_{\varkappa}(y^{\pm}_i)=0$, $i=1,2,3$, by (\ref{hl}).
Moreover,
$$
g(\widetilde a^{+},\widetilde
b^{+})=g(a^{+},b^{+})\circ\pi-\sum\limits_{i,j=1}^3 y^{+}_iy^{+}_j
(g(a^{+},s^{+}_i)\circ\pi)(g(b^{+},s^{+}_j)\circ\pi)
$$
Hence $X^h_{\varkappa}(g(\widetilde a^{+},\widetilde b^{+}))=0$.
Similarly $X^h_{\varkappa}(g(\widetilde a^{-},\widetilde b^{-}))=0$.
Therefore
$$
X^h_{\varkappa}(G_{\bf t}({\mathcal K}_{\nu}\widetilde a,\widetilde
b))=0.
$$
This proves (iv).

\smallskip

Next,
$$
(D_{U}F_{\bf t,\nu})(X^h,V)=U(G_{\bf t}({\mathcal
K}_{\nu}X^h,\widetilde b))-G_{\bf t}({\mathcal
K}_{\nu}D_{U}X^h,V)-G_{\bf t}({\mathcal K}_{\nu}X^h,D_{U}\widetilde
b)=0
$$
since ${\mathcal K}_{\nu}X^h$ and ${\mathcal K}_{\nu}D_{U}X^h$ are
horizontal vectors and $D_{U}\widetilde b $ is vertical. This is
identity (v).

\smallskip

Since $D=\nabla$ for vertical vector fields, identity (vi) is a
straightforward consequence from the definition of ${\mathcal
K}_{\nu}$ and the fact that $\nabla$ is a metric connection.

\end{proof}

Let $(N,h)$ be a Riemannian almost product manifold with almost
product structure $P$. Its Nijenhuis tensor is defined by
$$
{\mathcal N}_P(A,B)=[A,B]+[PA,PB]-P[PA,B]-P[A,PB]
$$
As usual, the structure $P$ is called integrable if the Nijenhuis
tensor vanishes. This condition is equivalent to the integrability
of both the vertical and horizontal distributions on the manifold
$N$. In this case $N$ is locally the product of two Riemannian
manifolds and $P$ is the trivial product structure determined by
these manifolds.

\smallskip

Denote by ${\mathcal N}_{\nu}$ the Nijenhuis tensor of the
endomorphism ${\mathcal K}_{\nu}$ of $T{\mathscr P}$. It can be
written in terms of the form $F_{\bf t,\nu}$ as
\begin{equation}\label{N-DF}
\begin{array}{c}
G_{\bf t}({\mathcal N}_{\nu}(A,B),C)=(D_{A}F_{\bf t,\nu})({\mathcal K}_{\nu}B,C) -(D_{B}F_{\bf t,\nu})({\mathcal K}_{\nu}A,C)\\[6pt]
+(D_{{\mathcal K}_{\nu}A}F_{\bf t,\nu})(B,C)-(D_{{\mathcal
K}_{\nu}B}F_{\bf t,\nu})(A,C).
\end{array}
\end{equation}

This identity, Corollary~\ref{cor-DFhhv}, and Lemma~\ref{D-F} imply:

\begin{cor}\label{N}
Let $\varkappa=(\sigma^{+},\sigma^{-})\in{\mathscr P}$, $X,Y\in
T_{\pi(\varkappa)}M$, $U,V\in{\mathcal V}_{\varkappa}$. Set
$P_{\varkappa}=K_{\sigma^{+}}\circ K_{\sigma^{-}}$. Then
$$
{\mathcal N}_{\nu}(X^h,Y^h)_{\varkappa}=R(X\wedge
Y+P_{\varkappa}X\wedge P_{\varkappa}Y)\varkappa-{\mathcal
K}_{\nu}(R(X\wedge P_{\varkappa}Y+ P_{\varkappa}X\wedge
Y)\varkappa);
$$
$$
{\mathcal
N}_{\nu}(X^h,U)_{\varkappa}=-((K_{\sigma^{+}}K_{U^{+}}+K_{\sigma^{-}}K_{U^{-}}+
\varepsilon_{\nu}K_{\sigma^{-}}K_{U^{+}}+\varepsilon_{\nu}(-1)^{\nu+1}K_{\sigma^{+}}K_{U^{-}})
X)^h_{\varkappa};
$$
$$
{\mathcal N}_{\nu}(U,V)=0.
$$
\end{cor}

\begin{prop}\label{non int} The almost  product structures ${\mathcal K}_{\nu}$ are never integrable.
\end{prop}

\begin{proof}
Take an oriented orthonormal basis $E_1,...,E_4$ of a tangent space
$T_pM$ and define $s_i^{\pm}$, $i=1,2,3$, by (\ref{s-basis}). Set
$\varkappa=(s_1^{+},s_1^{-})$, $U=(s_2^{+},0)$. Then ${\mathcal
N}_1(E_3,U)={\mathcal N}_2(E_3,U)=2(E_2)^h_{\varkappa}$ and
${\mathcal N}_3(E_1,U)={\mathcal N}_4(E_1,U)=-2(E_4)^h_{\varkappa}$.
\end{proof}

\section{Gil-Medrano conditions on the manifold  ${\mathscr
P}$}\label{Gil-Medrano}

Let $(N,h)$ be a Riemannian almost product manifold with almost
product structure $P$ and Levi-Civita connection $\nabla$. Let
$\mathfrak{D}$ be one of its vertical or horizontal distribution.
Denote the dimension of $\mathfrak{D}$ by $m$. Define an $1$-form on
$N$ setting
\begin{equation}\label{alpha}
\alpha(X)=\sum\limits_{l=1}^m h((\nabla_{E_l}P)(E_l),X), \quad X\in
T_pN,
\end{equation}
where $\{E_l\}$ is an orthonormal basis of $\mathfrak{D}_p$.

Following \cite{Med}, we shall say that:

\smallskip

\noindent $(a)$ $\mathfrak{D}$ has the property $F$ if
$(\nabla_{A}P)(B)=(\nabla_{B}P)(A)$ for every $A,B\in \mathfrak{D}$;

\smallskip

\noindent $(b)$ $\mathfrak{D}$ has the property $D_1$  if
$(\nabla_{A}P)(B)=-(\nabla_{B}P)(A)$ for  $A,B\in \mathfrak{D}$
(equivalently, $(\nabla_{A}P)(A)=0$);

\smallskip

\noindent $(c)$  $\mathfrak{D}$ has the property $D_2$  if
$\alpha(X)=0$ for every $X\in \mathfrak{D}^{\perp}$;

\smallskip

\noindent $(d)$  $\mathfrak{D}$ has the property $D_3$  if
$$
h((\nabla_{A}P)(B),X)+h((\nabla_{B}P)(A),X)=\frac{2}{m}h(A,B)\alpha(X),\>
A,B\in \mathfrak{D},\> X\in\mathfrak{D}^{\perp};
$$

\smallskip

\noindent $(e)$  $\mathfrak{D}$ has the property $F_i$, $i=1,2,3$,
if it has the properties $F$ and $D_i$.

\smallskip

\noindent {\bf Remark 2.} Note that  $\mathfrak{D}$ has the property
$D_1$ if and only if it has the properties $D_2$ and $D_3$.

\smallskip

For the geometric interpretations of these conditions given in
\cite{Med}, see  Section~\ref{Geom int}.

\smallskip

Combining conditions $F$, $D_i$, $F_i$ for the vertical and the
horizontal distributions on $(N,h)$, and eliminating their duality,
we obtain the 36 Naveira classes.

\begin{lemma}\label{DAFB}
Let $\varkappa=(\sigma^{+},\sigma^{-})\in{\mathscr P}$, $X,Y,Z\in
T_{\pi(\varkappa)}M$, and $U,V,W\in{\mathcal V}_{\varkappa}$. Set
$A=(X^h_{\varkappa}+U)+{\mathcal K}_{\nu}(X^h_{\varkappa}+U)$,
$B=(Y^h_{\varkappa}+V)+{\mathcal K}_{\nu}(Y^h_{\varkappa}+V)$. Then
$$
\begin{array}{lr}
G_{\bf t}((D_{A}{\mathcal K}_{\nu})(B),Z^h_{\varkappa})\\[8pt]
= -\frac{1}{2}g({\mathcal
R}([\varepsilon_{\nu}+1]t_1\sigma^{+}\times
V^{+}+[\varepsilon_{\nu}(-1)^{\nu}-1]t_2\sigma^{-}\times V^{-}),X\wedge Z\\[8pt]
\hspace{5.5cm}-X\wedge P_{\varkappa}Z +P_{\varkappa}X\wedge Z - P_{\varkappa}X\wedge P_{\varkappa}Z)\\[8pt]
-\frac{1}{2}g({\mathcal R}([\varepsilon_{\nu}+1]t_1\sigma^{+}\times
U^{+}+[\varepsilon_{\nu}(-1)^{\nu}-1]t_2\sigma^{-}\times U^{-}),Y\wedge Z\\[8pt]
\hspace{5.5cm}-Y\wedge P_{\varkappa}Z+P_{\varkappa}Y\wedge Z - P_{\varkappa}Y\wedge P_{\varkappa}Z)\\[8pt]
+g(([\varepsilon_{\nu}+1]K_{\sigma^{-}}K_{U^{+}}-[\varepsilon_{\nu}(-1)^{\nu}-1]K_{\sigma^{+}}K_{U^{-}})(Y+P_{\varkappa}Y),Z);
\end{array}
$$

$$
\begin{array}{lr}
G_{\bf t}((D_{A}{\mathcal K}_{\nu})(B),W)\\[8pt]
=-\frac{1}{2}g({\mathcal R}([\varepsilon_{\nu}-1]t_1\sigma^{+}\times
W^{+} + [\varepsilon_{\nu}(-1)^{\nu}+1]t_2\sigma^{-}\times
W^{-}),X\wedge Y\\[8pt]
\hspace{5.4cm}+X\wedge P_{\varkappa}Y +P_{\varkappa}X\wedge
Y+P_{\varkappa}X\wedge P_{\varkappa}Y),
\end{array}
$$
where $P_{\varkappa}=K_{\sigma^{+}}\circ K_{\sigma^{-}}$.
\end{lemma}

\begin{proof}
These formulas follow from Lemma~\ref{D-F} and the identity
$$
G_{\bf t}((D_{A}{\mathcal K}_{\nu})(B),C)=(D_{A}F_{\bf t,\nu})(B,C)
$$
by a simple computation.
\end{proof}

\bigskip

Let $\mathcal{D}_{\nu}$ be the distribution on the manifold
${\mathscr P}$ for which  ${\mathcal K}_{\nu}|\mathcal{D}_{\nu}
=Id$, $\nu=1,2$.

\begin{prop}\label{prop-D-F}
$(i)$ ~ The distribution ${\mathcal D}_{\nu}$ of the almost product
structure ${\mathcal K}_{\nu}$ does not have the property {\rm{F}}
for $\nu=1,2,4$.

 $(ii)$ ~ The distribution ${\mathcal D}_{3}$ has the property {\rm{F}} if and only if $(M,g)$ is of constant curvature.
\end{prop}

\begin{proof}
$(i)$ Let $E_1,...,E_4$ be an oriented orthonormal basis of a
tangent space $T_pM$. Define $s_i^{\pm}$, $i=1,2,3$, by
(\ref{s-basis}), and set $\varkappa=(s_1^{+},s_1^{-})$, $X=0$,
$U=(s_2^{+},s_2^{-})$, $Y=E_3$, $V=0$, $Z=E_2$. Then the identity
$G_{\bf t}((D_{A}{\mathcal K}_{\nu})(B),Z^h_{\varkappa})=G_{\bf
t}((D_{B}{\mathcal K}_{\nu})(A),Z^h_{\varkappa})$ becomes
$[\varepsilon_{\nu}+1]-[\varepsilon_{\nu}(-1)^{\nu}-1]=0$, an
identity, which does not hold for $\nu=1,2,4$.

$(ii)$ By Lemma~\ref{DAFB}, the distribution ${\mathcal D}_{3}$ has
the property \rm{F} if and only if
\begin{equation}\label{eq D3-1}
g({\mathcal R}(t_1\sigma^{+}\times W^{+} - t_2\sigma^{-}\times
W^{-}),X\wedge Y+X\wedge P_{\varkappa}Y +P_{\varkappa}X\wedge
Y+P_{\varkappa}X\wedge P_{\varkappa}Y)=0
\end{equation}
for every $\varkappa=(\sigma^{+},\sigma^{-})\in{\mathscr P}$,
$W^{\pm}\in\Lambda^{2}_{\pm}T_{\pi(\varkappa)}M$ with
$W^{\pm}\perp\sigma^{\pm}$ and  $X,Y\in T_{\pi(\varkappa)}M$.
Applying this identity  for $(W^{+},-W^{-})$, we see that condition
(\ref{eq D3-1})  is equivalent to
$$
g({\mathcal R}(\sigma^{\pm}\times W^{\pm}),X\wedge Y+X\wedge
P_{\varkappa}Y +P_{\varkappa}X\wedge Y+P_{\varkappa}X\wedge
P_{\varkappa}Y)=0.
$$
Replacing $\sigma^{-}$ and $W^{-}$ by $-\sigma^{-}$ and $-W^{-}$, we
observe that the latter equations are equivalent to
\begin{equation}\label{eq D3-2}
g({\mathcal R}(\sigma^{\pm}\times W^{\pm}),X\wedge
Y+P_{\varkappa}X\wedge P_{\varkappa}Y)=0.
\end{equation}
Let $E_1,...,E_4$ be an oriented orthonormal basis of a tangent
space $T_pM$ and define $s_i^{\pm}$, $i=1,2,3$, by (\ref{s-basis}).
We apply(\ref{eq D3-2}) with (a) $\varkappa=(s_1^{+},s_1^{-})$,
$W^{+}=s_2^{+}, s_3^{+}$, $(X,Y)=(E_1,E_2), (E_3,E_4)$, (b)
$\varkappa=(s_1^{+},s_2^{-})$, $W^{+}=s_3^{+}$, $(X,Y)=(E_1,E_3)$,
(c) $\varkappa=(s_3^{+},s_1^{-})$, $W^{+}=s_2^{+}$,
$(X,Y)=(E_1,E_2)$. This gives
$$
\begin{array}{c}
g({\mathcal R}(s_3^{+}),s_1^{+})=g({\mathcal
R}(s_3^{+}),s_1^{-})=g({\mathcal R}(s_2^{+}),s_1^{+})
=g({\mathcal R}(s_2^{+}),s_1^{-})=0,\\[8pt]
g({\mathcal R}(s_2^{+}),s_2^{+})=0,\quad g({\mathcal
R}(s_1^{+}),s_1^{-})=0.
\end{array}
$$
Replacing the basis $E_1,E_2,E_3,E_4$ by $E_1,E_3,E_4,E_2$ and
$E_1,E_4,E_2,E_3$, we see that
$$
g({\mathcal R}(s_i^{+}),s_j^{+})=g({\mathcal
R}(s_i^{+}),s_j^{-})=0,\quad i,j=1,2,3.
$$
Therefore ${\mathcal W}_{+}={\mathcal B}=0$. In the same way, we get
${\mathcal W}_{-}=0$ from (\ref{eq D3-2}). This shows that $(M,g)$
is of constant curvature.

Conversely, if  $(M,g)$ is of constant curvature,  the identity
(\ref{eq D3-2}) is satisfied by (\ref{com-anticom}).

\end{proof}

\smallskip




\begin{prop}\label{prop-D-1}
{\rm{(I)}}~The distribution ${\mathcal D}_{\nu}$, $\nu=1,2,4$, of
the almost product structure ${\mathcal K}_{\nu}$ has the property
{\rm{$D_1$}} if and only if:

$(i)$~ $(M,g)$ is of  positive constant sectional curvature $\chi$
and $t_1=t_2=\displaystyle{\frac{3}{8\chi}}$, in the case  $\nu=1$;

\smallskip

$(ii)$~ $(M,g)$ is anti-self-dual and Einstein with positive  scalar
curvature $s$, and $t_1=\displaystyle{\frac{6}{s}}$, in the case
$\nu=2$  (no condition on $t_2>0$).

\smallskip

$(iii)$ ~ $(M,g)$ is self-dual and Einstein with positive  scalar
curvature $s$, and $t_2=\displaystyle{\frac{6}{s}}$, in the case
$\nu=4$  (no condition on $t_1>0$).

\smallskip

{\rm{(II)}} ~ The distribution ${\mathcal D}_{3}$ has the property
{\rm{$D_1$}}.

\end{prop}

\begin{proof}

By Lemma~\ref{DAFB},  ${\mathcal D}_{\nu}$ has the property  \rm
{$D_1$} if and only if
\begin{equation}\label{D-1-eq1}
\begin{array}{l}
g({\mathcal R}([\varepsilon_{\nu}+1]t_1\sigma^{+}\times
U^{+}+[\varepsilon_{\nu}(-1)^{\nu}-1]t_2\sigma^{-}\times
U^{-}),X\wedge Z-X\wedge
P_{\varkappa}Z\\[8pt]
\hspace{7.8cm}+P_{\varkappa}X\wedge Z -
P_{\varkappa}X\wedge P_{\varkappa}Z)\\[8pt]
-g(([\varepsilon_{\nu}+1]K_{\sigma^{-}}K_{U^{+}}-[\varepsilon_{\nu}(-1)^{\nu}-1]K_{\sigma^{+}}K_{U^{-}})(X+P_{\varkappa}X),Z)\\[6pt]
=0
\end{array}
\end{equation}
for every $\varkappa=(\sigma^{+},\sigma^{-})\in{\mathscr P}$,
$U^{\pm}\in \Lambda^2_{\pm}T_{\pi(\varkappa)}M$ with
$U^{\pm}\perp\sigma^{\pm}$ and  $X,Z\in T_{\pi(\varkappa)}M$.

As in the proof of the preceding proposition, it is easy to see that
this condition is equivalent to the identities
\begin{equation}\label{D-1}
\begin{array}{l}
[\varepsilon_{\nu}+1]\big\{t_1g({\mathcal R}(\sigma^{+}\times U^{+}),X\wedge Z-P_{\varkappa}X\wedge P_{\varkappa}Z)\\[8pt]
\hspace{8.1cm}-g(K_{\sigma^{+}}K_{U^{+}}X,Z)\big\}=0,\\[8pt]
[\varepsilon_{\nu}(-1)^{\nu}-1]\big\{t_2g({\mathcal
R}(\sigma^{-}\times U^{-}),X\wedge Z-P_{\varkappa}X\wedge
P_{\varkappa}Z)\\[8pt]
\hspace{8.1cm} +g(K_{\sigma^{-}}K_{U^{-}}X,Z)\big\}=0.
\end{array}
\end{equation}

Clearly both identities are satisfied if $\nu=3$. Note also that, by
(\ref{KK}), $K_{\sigma^{+}}K_{U^{+}}=K_{\sigma^{+}\times U^{+}}$ and
$K_{\sigma^{-}}K_{U^{-}}=-K_{\sigma^{-}\times U^{-}}$. Thus, if
$\nu=1$, changing the orientation of $M$ interchanges the identities
in (\ref{D-1}). If $\nu=2$, the second identity in (\ref{D-1}) is
trivially satisfied and if $\nu=4$, so does the second one.

Now, suppose that $\varepsilon_{\nu}+1\neq 0$ and the first identity
in (\ref{D-1}) holds. Let $E_1,...,E_4$ be an oriented orthonormal
basis of a tangent space $T_pM$ of $M$ and define $s^{\pm}_i$,
$i=1,2,3$, by (\ref{s-basis}). Taking $\varkappa=(s^{+}_1,s^{-}_1)$,
$U^{+}=s^{+}_2$, we get from the first identity of (\ref{D-1})
$$
\begin{array}{c}
g({\mathcal R}(s^{+}_3),E_1\wedge E_3)=g({\mathcal R}(s^{+}_3),E_2\wedge E_4)=0,\\[6pt]
2t_1g({\mathcal R}(s^{+}_3),E_1\wedge E_4)-1=0,\quad 2t_1g({\mathcal
R}(s^{+}_3),E_2\wedge E_3)-1=0,
\end{array}
$$
Therefore
\begin{equation}\label{D-1-1}
\begin{array}{c}
g({\mathcal R}(s^{+}_3),s^{+}_2)=g({\mathcal R}(s^{+}_3),s^{-}_2)=g({\mathcal R}(s^{+}_3),s^{-}_3)=0,\\[6pt]
t_1g({\mathcal R}(s^{+}_3),s^{+}_3)-1=0.
\end{array}
\end{equation}
Similarly, taking $U^{+}=s^{+}_3$, we obtain
\begin{equation}\label{D-1-2}
\begin{array}{c}
g({\mathcal R}(s^{+}_2),s^{+}_3)=g({\mathcal R}(s^{+}_2),s^{-}_2)=g({\mathcal R}(s^{+}_2),s^{-}_3)=0,\\[6pt]
t_1g({\mathcal R}(s^{+}_2),s^{+}_2)-1=0.
\end{array}
\end{equation}
Replacing the basis $E_1,E_2,E_3,E_4$ by $E_1,E_3,E_4,E_2$ and
$E_1,E_4,E_2,E_3$, we see from (\ref{D-1-1}) and (\ref{D-1-2}) that
\begin{equation}\label{D-1-3}
t_1g({\mathcal R}(s^{+}_i),s^{+}_j)-\delta_{ij}=0,\quad g({\mathcal
R}(s^{+}_i),s^{-}_j)=0,\quad i,j=1,2,3.
\end{equation}
Now, the curvature decomposition (\ref{dec}) and the fact that
$Trace\,{\mathcal W}_{+}=0$ imply $\displaystyle{t_1=\frac{6}{s}}$.
Then the first identity of (\ref{D-1-3}) gives $g({\mathcal
W}_{+}(s_i^{+}),s_i^{+})=0$, $i=1,2,3$. Hence ${\mathcal W}_{+}=0$.
The second identity of (\ref{D-1-3}) means that ${\mathcal B}=0$.

Conversely, if  $t_1=\displaystyle{\frac{6}{s}}$ and ${\mathcal
B}={\mathcal W}_{+}=0$,  it is easy to check, using (\ref{cs}),
(\ref{com-anticom})  and (\ref{KK}), that the first identity of
(\ref{D-1}) is  fulfilled. This proves the result for $\nu=2$.

If $\nu=1$ or $\nu=4$, the second identity of (\ref{D-1}) holds if
and only if $t_2=\displaystyle{\frac{6}{s}}$ and ${\mathcal
B}={\mathcal W}_{-}=0$.

\end{proof}



\begin{prop}\label{D-D2}
 The distribution ${\mathcal D}_{\nu}$, $\nu=1,...,4$,  has the property \rm{$D_2$}.
\end{prop}

\begin{proof}
Denote by $\alpha_{\nu}$  the $1$-form corresponding to the
distribution ${\mathcal D}_{\nu}$ via (\ref{alpha}). Let
$\varkappa=(\sigma^{+},\sigma^{-})\in{\mathscr P}$ and set
$P_{\varkappa}=K_{\sigma^{+}}\circ K_{\sigma^{-}}$. Take an oriented
orthonormal basis $E_1,...,E_4$ of $T_{\pi(\varkappa)}M$ such that
$P_{\varkappa}E_i=-E_i$ for $i=1,2$ and $P_{\varkappa}E_j=E_j$ for
$j=3,4$. Let $V^{\pm}_i$, $i=1,2$, be a $g$-orthonormal basis of
${\mathcal V}_{\pm}$.  Then $E^h_3,E^h_4,\frac{1}{\sqrt
t_1}(V^{+}_1,0), \frac{1}{\sqrt t_1}(V^{+}_2,0), \frac{1}{\sqrt
t_2}(0,V^{-}_1), \frac{1}{\sqrt t_2}(0,V^{-}_2)$ is a $G_{\bf
t}$-orthonormal  basis of the fibre of ${\mathcal D}_{1}$ at
$\pi({\varkappa})$, $E^h_3,E^h_4,\frac{1}{\sqrt t_1}(V^{+}_1,0),
\frac{1}{\sqrt t_1}(V^{+}_2,0)$ of ${\mathcal D}_{2}$, $E^h_3,E^h_4$
of ${\mathcal D}_{3}$, and $E^h_3,E^h_4, \frac{1}{\sqrt
t_2}(0,V^{-}_1), \frac{1}{\sqrt t_2}(0,V^{-}_2)$ is a $G_{\bf
t}$-orthonormal basis of the fibre of ${\mathcal D}_{4}$ . Using
these bases, we get $\alpha_{\nu}=0$ by Lemma~\ref{D-F}.
\end{proof}

Remark 2 and Proposition~\ref{D-D2} imply:

\begin{prop}\label{prop-D-3}
The distribution ${\mathcal D}_{\nu}$, $\nu=1,...,4$,  has the
property {\rm{$D_3$}} exactly when it has the property \rm{$D_1$}.
\end{prop}

Lemma~\ref{D-F} imply the following.

\begin{lemma}\label{DtilAFtilB}
Let $\varkappa=(\sigma^{+},\sigma^{-})\in{\mathscr P}$, $X,Y,Z\in
T_{\pi(\varkappa)}M$, and $U,V,W\in{\mathcal V}_{\varkappa}$. Set
$\widetilde A=(X^h_{\varkappa}+U)-{\mathcal
K}_{\nu}(X^h_{\varkappa}+U)$, $\widetilde
B=(Y^h_{\varkappa}+V)-{\mathcal K}_{\nu}(Y^h_{\varkappa}+V)$. Then

$$
\begin{array}{l}
G_{\bf t}((D_{\widetilde A}{\mathcal K}_{\nu})(\widetilde B),Z^h_{\varkappa})\\[8pt]
= -\frac{1}{2}g({\mathcal
R}([\varepsilon_{\nu}-1]t_1\sigma^{+}\times
V^{+}+[\varepsilon_{\nu}(-1)^{\nu}+1]t_2\sigma^{-}\times
V^{-}),X\wedge Z\\[8pt]
\hspace{5.4cm}+X\wedge P_{\varkappa}Z -P_{\varkappa}X\wedge Z - P_{\varkappa}X\wedge P_{\varkappa}Z)\\[8pt]
-\frac{1}{2}g({\mathcal R}([\varepsilon_{\nu}-1]t_1\sigma^{+}\times
U^{+}+[\varepsilon_{\nu}(-1)^{\nu}+1]t_2\sigma^{-}\times
U^{-}),Y\wedge Z \\[8pt]
\hspace{5.4cm}+Y\wedge P_{\varkappa}Z -P_{\varkappa}Y\wedge Z -
P_{\varkappa}Y\wedge P_{\varkappa}Z)\\[8pt]
-g(([\varepsilon_{\nu}-1]K_{\sigma^{-}}K_{U^{+}}-
[\varepsilon_{\nu}(-1)^{\nu}+1]K_{\sigma^{+}}K_{U^{-}})(Y-P_{\varkappa}Y),Z);
\end{array}
$$

$$
\begin{array}{l}
G_{\bf t}((D_{\widetilde A}{\mathcal K}_{\nu})(\widetilde B),W)\\[8pt]
=-\frac{1}{2}g({\mathcal R}([\varepsilon_{\nu}+1]t_1\sigma^{+}\times
W^{+}+[\varepsilon_{\nu}(-1)^{\nu}-1]t_2\sigma^{-}\times W^{-}),X\wedge Y\\[8pt]
\hspace{5.4cm}-X\wedge P_{\varkappa}Y-P_{\varkappa}X\wedge
Y+P_{\varkappa}X\wedge P_{\varkappa}Y),
\end{array}
$$
where $P_{\varkappa}=K_{\sigma^{+}}\circ K_{\sigma^{-}}$.
\end{lemma}

\begin{prop}\label{prop D-perp-F}
$(i)$ ~ The distribution ${\mathcal D}^{\perp}_{1}$ has the property
{\rm{F}} if and only if the manifold $(M,g)$ is of constant
curvature.

\noindent $(ii)$ ~ The distribution ${\mathcal D}^{\perp}_{\nu}$
does not have the property {\rm{F}} for $\nu=2,3,4$.
\end{prop}

\begin{proof}
By Lemma~\ref{DtilAFtilB}, the distribution ${\mathcal
D}^{\perp}_{\nu}$ has the property {\rm{F}} if and only if the
following two identities hold:
\begin{equation}\label{perp-F-eq1}
\begin{array}{c}
g(([\varepsilon_{\nu}-1]K_{\sigma^{-}}K_{U^{+}}-
[\varepsilon_{\nu}(-1)^{\nu}+1]K_{\sigma^{+}}K_{U^{-}})(Y-P_{\varkappa}Y),Z)\\[8pt]
=g(([\varepsilon_{\nu}-1]K_{\sigma^{-}}K_{V^{+}}-
[\varepsilon_{\nu}(-1)^{\nu}+1]K_{\sigma^{+}}K_{V^{-}})(X-P_{\varkappa}X),Z),\\[8pt]
g({\mathcal R}([\varepsilon_{\nu}+1]t_1\sigma^{+}\times
W^{+}+[\varepsilon_{\nu}(-1)^{\nu}-1]t_2\sigma^{-}\times
W^{-}),X\wedge Y-X\wedge
P_{\varkappa}Y\\[8pt]
\hspace{7cm}-P_{\varkappa}X\wedge Y+P_{\varkappa}X\wedge
P_{\varkappa}Y)=0
\end{array}
\end{equation}
for every $\varkappa=(\sigma^{+},\sigma^{-})\in{\mathscr P}$,
$X,Y,Z\in T_{\pi(\varkappa)}M$, $U,V,W\in{\mathcal V}_{\varkappa}$.

Let $E_1,...,E_4$ be an oriented orthonormal basis of a tangent
space $T_{p}M$ and define $s_i^{\pm}$, $i=1,2,3$, by
(\ref{s-basis}).

If $\nu=2,3,4$, the first identity of (\ref{perp-F-eq1}) does not
hold for $\varkappa=(s_1^{+},s_1^{-})$, $X=0$,
$U=(s_2^{+},s_2^{-})$, $Y=E_2$, $Z=E_3$.

If $\nu=1$,  (\ref{perp-F-eq1}) reduces to
$$
\begin{array}{l}
g({\mathcal R}(t_1\sigma^{+}\times W^{+}-t_2\sigma^{-}\times
W^{-}),X\wedge Y-X\wedge
P_{\varkappa}Y\\[8pt]
\hspace{6.7cm}-P_{\varkappa}X\wedge Y+P_{\varkappa}X\wedge
P_{\varkappa}Y)=0.
\end{array}
$$
This is equivalent to the identities
$$
g({\mathcal R}(\sigma^{\pm}\times W^{\pm}), X\wedge
Y+P_{\varkappa}X\wedge P_{\varkappa}Y)=0.
$$
As we have seen in the proof of Proposition~ \ref{prop-D-F},  the
latter identities are satisfied if and only if the manifold $(M,g)$
is of constant curvature.

\end{proof}

\begin{prop}\label{prop D-perp-1}
{\rm{(I)}}~The distribution ${\mathcal D}_{\nu}^{\perp}$,
$\nu=2,3,4$, of the almost product structure ${\mathcal K}_{\nu}$
has the property {\rm{$D_1$}} if and only if:

$(i)$~  $(M,g)$ is self-dual and Einstein  with positive  scalar
curvature $s$, and $t_2=\displaystyle{\frac{6}{s}}$, in the case
$\nu=2$  (no condition on $t_1>0$);

$(ii)$~  $(M,g)$ is of  positive constant sectional curvature $\chi$
and $t_1=t_2=\displaystyle{\frac{3}{8\chi}}$, in the case $\nu=3$;

\smallskip

$(iii)$~ $(M,g)$ is anti-self-dual and Einstein  with positive
scalar curvature $s$, and $t_1=\displaystyle{\frac{6}{s}}$, in the
case $\nu=4$  (no condition on $t_2>0$).

\smallskip

{\rm{(II)}}~ The distribution ${\mathcal D}_{1}^{\perp}$ has the
property {\rm{$D_1$}}.

\end{prop}

\begin{proof}
By Lemma~\ref{DtilAFtilB},  ${\mathcal D}_{\nu}^{\perp}$ has the
property  \rm {$D_1$} if and only if
\begin{equation}\label{perp-D-1}
\begin{array}{l}
g({\mathcal R}([\varepsilon_{\nu}-1]t_1\sigma^{+}\times
U^{+}+[\varepsilon_{\nu}(-1)^{\nu}+1]t_2\sigma^{-}\times
U^{-}),X\wedge Z+X\wedge
P_{\varkappa}Z\\[8pt]
\hspace{7.8cm}-P_{\varkappa}X\wedge Z -
P_{\varkappa}X\wedge P_{\varkappa}Z)\\[8pt]
+g(([\varepsilon_{\nu}-1]K_{\sigma^{-}}K_{U^{+}}-
[\varepsilon_{\nu}(-1)^{\nu}+1]K_{\sigma^{+}}K_{U^{-}})(X-P_{\varkappa}X),Z)=0
\end{array}
\end{equation}
for every $\varkappa=(\sigma^{+},\sigma^{-})\in{\mathscr P}$,
$U^{\pm}\in \Lambda^2_{\pm}T_{\pi(\varkappa)}M$ with
$U^{\pm}\perp\sigma^{\pm}$, and  $X,Z\in T_{\pi(\varkappa)}M$.

This condition is equivalent to
$$
\begin{array}{l}
[\varepsilon_{\nu}-1]\big\{t_1g({\mathcal R}(\sigma^{+}\times U^{+}),X\wedge Z-P_{\varkappa}X\wedge P_{\varkappa}Z)\\[8pt]
\hspace{8.1cm}-g(K_{\sigma^{+}}K_{U^{+}}X,Z)\big\}=0,\\[8pt]
[\varepsilon_{\nu}(-1)^{\nu}+1]\big\{t_2g({\mathcal
R}(\sigma^{-}\times U^{-}),X\wedge Z-P_{\varkappa}X\wedge
P_{\varkappa}Z)\\[8pt]
\hspace{8.1cm} +g(K_{\sigma^{-}}K_{U^{-}}X,Z)\big\}=0.
\end{array}
$$
Obviously, if $\nu=1$ these conditions are satisfied, if $\nu=2$ the
first identity trivially hods and if $\nu=4$ the second one holds.
The result follows from the proof of Proposition~\ref{prop-D-1}.

\end{proof}

Lemma~\ref{D-F} easily implies:

\begin{prop}
 The distribution ${\mathcal D}_{\nu}^{\perp}$  has the property \rm{$D_2$}.
\end{prop}

\begin{prop}
The distribution ${\mathcal D}_{\nu}^{\perp}$  has the property
{\rm{$D_3$}} exactly when it has the property \rm{$D_1$}.
\end{prop}

\section {The Naveira classes of the manifold  ${\mathscr P}$}

The results in the preceding section  allow one to determine the
Naveira classes of $({\mathscr P},{\mathcal K}_{\nu}, G_{\bf t})$,
${\bf t}=(t_1,t_2)$, as follows.

\medskip

\medskip

\begin{thm}\label{K1}
The Riemannian almost product manifold $({\mathscr P},{\mathcal
K}_{1}, G_{\bf t})$ belongs to the Naveira class ${\mathscr
W}_1\oplus {\mathscr W}_2\oplus {\mathscr W}_4\oplus {\mathscr W}_5$
or to the class ${\mathscr W}_1\oplus {\mathscr W}_4$.

\smallskip

$({\mathscr P},{\mathcal K}_{1},G_{\bf t})\in {\mathscr W}_1\oplus
{\mathscr W}_4$ if and only if $(M,g)$ is of positive constant
curvature $\chi$ and $t_1=t_2=\frac{3}{8\chi}$.

\end{thm}

\begin{thm}\label{K2}
The Riemannian almost product manifold $({\mathscr P},{\mathcal
K}_{2}, G_{\bf t})$ belongs to the Naveira classes ${\mathscr
W}_1\oplus {\mathscr W}_2\oplus {\mathscr W}_4\oplus {\mathscr
W}_5$, ${\mathscr W}_1\oplus {\mathscr W}_4\oplus {\mathscr W}_5$,
or ${\mathscr W}_1\oplus {\mathscr W}_4$.

\smallskip

$(i)$~$({\mathscr P},{\mathcal K}_{2}, G_{\bf t})\in {\mathscr
W}_1\oplus {\mathscr W}_4\oplus {\mathscr W}_5$ if and only if
$(M,g)$ is anti-self-dual and Einstein with positive  scalar
curvature $s$, and $t_1=\displaystyle{\frac{6}{s}}$ (no condition on
$t_2>0$).

\smallskip

$(ii)$~ $({\mathscr P},{\mathcal K}_{2},G_{\bf t})\in {\mathscr
W}_1\oplus {\mathscr W}_4$ if and only if $(M,g)$ is of positive
constant curvature $\chi$ and $t_1=t_2=\frac{3}{8\chi}$.

\end{thm}

\begin{thm}\label{K3}
 The Riemannian almost product manifold $({\mathscr P},{\mathcal K}_{3}, G_{\bf t})$ belongs to the Naveira classes
${\mathcal W}_1\oplus{\mathcal W}_4\oplus{\mathcal W}_5$, ${\mathcal
W}_4\oplus{\mathcal W}_5$, or ${\mathcal W}_4$.

\smallskip

$(i)$~ $({\mathscr P},{\mathcal K}_{3}, G_{\bf t})\in {\mathcal
W}_4\oplus{\mathcal W}_5$ if and only if $(M,g)$ is of constat
curvature.

\smallskip

$(ii)$~$({\mathscr P},{\mathcal K}_{3}, G_{\bf t})\in {\mathcal
W}_4$  if and only if $(M,g)$ is of  positive constant sectional
curvature $\chi$ and $t_1=t_2=\displaystyle{\frac{3}{8\chi}}$

\end{thm}

 \begin{thm}\label{K4}
 The Riemannian almost product manifold $({\mathscr P},{\mathcal K}_{4}, G_{\bf t})$ belongs to the Naveira classes
 ${\mathscr W}_1\oplus {\mathscr W}_2\oplus {\mathscr W}_4\oplus {\mathscr W}_5$,
 ${\mathscr W}_1\oplus {\mathscr W}_4\oplus {\mathscr W}_5$, or
 ${\mathscr W}_1\oplus {\mathscr W}_4$.

 \smallskip

 $(i)$~$({\mathscr P},{\mathcal K}_{4}, G_{\bf t})\in {\mathscr W}_1\oplus {\mathscr W}_4\oplus {\mathscr W}_5$ if and
 only if  $(M,g)$ is self-dual and Einstein with positive  scalar curvature $s$, and $t_2=\displaystyle{\frac{6}{s}}$
 (no condition on $t_1>0$).

 \smallskip

 $(ii)$~ $({\mathscr P},{\mathcal K}_{4},G_{\bf t})\in {\mathscr W}_1\oplus {\mathscr W}_4$ if and only if $(M,g)$ is of
 positive constant curvature $\chi$ and $t_1=t_2=\frac{3}{8\chi}$.

 \end{thm}

 \section{Geometric interpretation}\label{Geom int}

  In this section, we restate the results obtained in preceding sections in geometric terms.

Recall the geometric characterizations of the Gil-Medrano conditions
for  the  vertical or horizontal distribution $\mathfrak{D}$  of a
Riemannian almost product manifold (\cite{Med}).  First, condition
$F$ is equivalent to  $\mathfrak{D}$ being integrable. Next, a
second fundamental form $T$  of a distribution on a Riemannian
manifold has been proposed by B. Reinhart in \cite{Rein}. It is a
symmetric $2$-form with values in the normal bundle. If the
distribution is integrable, $T$ coincides with the usual second
fundamental form of the leaves as immersed submanifolds. A
distribution is called minimal if the trace of $T$ vanishes; it is
called totally geodesic if $T=0$. It has been proved in \cite{Med}
and \cite{Rein} that a distribution is totally  geodesic if and only
if every geodesic, which is tangent to the distribution at one
point, is tangent to it at all points. Now, condition $D_1$ means
that $\mathfrak{D}$ is totally geodesic, while condition $D_2$ is
equivalent to $\mathfrak{D}$ being minimal.

 \begin{thm}\label{all for D}
 {\rm{(I)}}~ $(i)$~ The distributions ${\mathcal D}_{\nu}$, $\nu=1,2,4$, are not integrable.

 \smallskip

 $(ii)$~ The distribution ${\mathcal D}_{3}$ is integrable if and only if $(M,g)$ is of constat curvature.

\medskip

{\rm{(II)}}~All distributions ${\mathcal D}_{\nu}$ are minimal,
$\nu=1,...,4$.

\medskip

{\rm{(III)}}~ $(i)$ ~The distribution ${\mathcal D}_{3}$ is totally
geodesic.

\smallskip

$(ii)$~ The distribution ${\mathcal D}_{\nu}$, $\nu=1,2,4$, is
totally geodesic if and only if:

\smallskip

 $\bullet$ ~$(M,g)$ is of  positive constant sectional curvature $\chi$ and $t_1=t_2=\displaystyle{\frac{3}{8\chi}}$, in the
case  $\nu=1$;

\smallskip

$\bullet$ ~$(M,g)$ is anti-self-dual and Einstein with positive
scalar curvature $s$, and $t_1=\displaystyle{\frac{6}{s}}$, in the
case $\nu=2$  (no condition on $t_2>0$).

\smallskip

$\bullet$ ~ $(M,g)$ is self-dual and Einstein with positive  scalar
curvature $s$, and $t_2=\displaystyle{\frac{6}{s}}$, in the case
$\nu=4$  (no condition on $t_1>0$).

 \end{thm}

  \begin{thm}\label{all for D-perp}
  {\rm{(I)}}~ $(i)$~ The distributions ${\mathcal D}_{\nu}^{\perp}$, $\nu=2,3,4$, are not integrable.

  \smallskip

  $(ii)$~ The distribution ${\mathcal D}_{1}^{\perp}$ is integrable if and only $(M,g)$ is of constat curvature.

 \medskip

 {\rm{(II)}}~All distributions ${\mathcal D}_{\nu}^{\perp}$ are minimal, $\nu=1,...,4$.

 \medskip

 {\rm{(III)}}~ $(i)$ ~ The distribution ${\mathcal D}_{1}^{\perp}$ is totally geodesic.

 \smallskip

 $(ii)$~ The distribution ${\mathcal D}_{\nu}^{\perp}$, $\nu=2,3,4$, is totally geodesic if and only if:

\smallskip

 $\bullet$ ~$(M,g)$ is  self-dual and Einstein with positive  scalar curvature $s$, and $t_2=\displaystyle{\frac{6}{s}}$, in
 the case $\nu=2$  (no condition on $t_1>0$).

 \smallskip
                                                                                                                                 \smallskip

 $\bullet$ ~$(M,g)$ is of  positive constant sectional curvature $\chi$ and $t_1=t_2=\displaystyle{\frac{3}{8\chi}}$, in the
 case  $\nu=3$;

 \smallskip

 $\bullet$ ~ $(M,g)$ is anti-self-dual and Einstein with positive  scalar curvature $s$, and $t_1=\displaystyle{\frac{6}{s}}$, in the
 case $\nu=4$  (no condition on $t_2>0$).

  \end{thm}

In the next theorem, we give a geometric interpretation of the
Naveira classes of the Riemannian almost product manifolds
$({\mathscr P},{\mathcal K}_{\nu}, G_{\bf t})$ determined in
Theorems~\ref{K1}-\ref{K4}.

\begin{thm}\label{all for Nav}

 {\rm{(I)}}~$(i)$~ The distributions ${\mathcal D}_{1}$ and ${\mathcal D}_{1}^{\perp}$ are both minimal \rm{($({\mathcal K_1},G_{\bf t})\in
 {\mathscr W}_1\oplus {\mathscr W}_2\oplus {\mathscr W}_4\oplus {\mathscr W}_5$)}.

\smallskip

$(ii)$~${\mathcal D}_{1}$ and ${\mathcal D}_{1}^{\perp}$ are totally
geodesic distributions \rm{($({\mathcal K_1},G_{\bf t})\in
{\mathscr W}_1\oplus {\mathscr W}_4$)} if and only if  $(M,g)$ is of
positive constant curvature $\chi$ and $t_1=t_2=\frac{3}{8\chi}$.

\medskip

{\rm{(II)}}~$(i)$~ The distributions ${\mathcal D}_{2}$ and
${\mathcal D}_{2}^{\perp}$ are both minimal \rm{($({\mathcal
K_2},G_{\bf t})\in
 {\mathscr W}_1\oplus {\mathscr W}_2\oplus {\mathscr W}_4\oplus {\mathscr W}_5$)}.

\smallskip

$(ii)$~  ${\mathcal D}_{2}$  is totally geodesic and ${\mathcal
D}_{2}^{\perp}$ is minimal \rm{($({\mathcal K}_{2}, G_{\bf t})\in
{\mathscr W}_1\oplus {\mathscr W}_4\oplus {\mathscr W}_5$)} if and
only if $(M,g)$ is  anti-self-dual and Einstein with positive
scalar curvature $s$, and $t_1=\displaystyle{\frac{6}{s}}$ (no
condition on $t_2>0$).

\smallskip

$(iii)$~  ${\mathcal D}_{2}$  and ${\mathcal D}_{2}^{\perp}$  are
totally geodesic distributions \rm{($({\mathcal K}_{2},G_{\bf t})\in
{\mathscr W}_1\oplus {\mathscr W}_4)$}  if and only if $(M,g)$ is of
positive constant curvature $\chi$ and $t_1=t_2=\frac{3}{8\chi}$.

\medskip

{\rm{(III)}}~$(i)$~The distribution ${\mathcal D}_{3}$ is totally
geodesic and ${\mathcal D}_{3}^{\perp}$ is minimal

\rm{((${\mathcal K_{3}},G_{\bf t})\in {\mathcal W}_1\oplus{\mathcal
W}_4\oplus{\mathcal W}_5$)}.

\smallskip

$(ii)$~ ${\mathcal D}_{3}$ is integrable and totally geodesic, and
${\mathcal D}_{3}^{\perp}$ is minimal \rm{($({\mathcal K}_{3},
G_{\bf t})\in {\mathcal W}_4\oplus{\mathcal W}_5$)} if and only if
$(M,g)$ is of constat curvature

\smallskip

$(iii)$~${\mathcal D}_{3}$ is integrable and totally geodesic, and
${\mathcal D}_{3}^{\perp}$ is totally geodesic \rm{($({\mathcal
K}_{3}, G_{\bf t})\in {\mathcal W}_4$)}  if and only if $(M,g)$ is
of  positive constant sectional curvature $\chi$ and
$t_1=t_2=\displaystyle{\frac{3}{8\chi}}$

\medskip

{\rm{(IV)}}~$(i)$~ The distributions ${\mathcal D}_{4}$ and
${\mathcal D}_{4}^{\perp}$ are both minimal \rm{($({\mathcal
K_4},G_{\bf t})\in
 {\mathscr W}_1\oplus {\mathscr W}_2\oplus {\mathscr W}_4\oplus {\mathscr W}_5$)}.

\smallskip

$(ii)$~  ${\mathcal D}_{4}$  is totally geodesic and ${\mathcal
D}_{4}^{\perp}$ is minimal \rm{($({\mathcal K}_{4}, G_{\bf t})\in
{\mathscr W}_1\oplus {\mathscr W}_4\oplus {\mathscr W}_5$)} if and
only if $(M,g)$ is self-dual and Einstein with positive  scalar
curvature $s$, and $t_2=\displaystyle{\frac{6}{s}}$ (no condition on
$t_1>0$).

\smallskip

$(iii)$~  ${\mathcal D}_{4}$  and ${\mathcal D}_{4}^{\perp}$  are
totally geodesic distributions \rm{($({\mathcal K}_{4},G_{\bf t})\in
{\mathscr W}_1\oplus {\mathscr W}_4)$}  if and only if $(M,g)$ is of
positive constant curvature $\chi$ and $t_1=t_2=\frac{3}{8\chi}$.

\end{thm}

\smallskip

It is a result of Hitchin (see \cite[Theorem 13.30]{Besse}) that
every compact self-dual (anti-self-dual) Einstein manifold with
positive scalar curvature is isometric to $S^4$ or ${\mathbb
C}{\mathbb P}^2$ with their standard metrics and orientations (resp.
the opposite orientations) (cf. also \cite{FK,Hit}).

It is well known \cite{AHS} that the twistor spaces ${\mathcal
Z}_{\pm}$ of $S^4$ and ${\mathbb C}{\mathbb P}^2$ can be identified
as smooth manifolds with ${\mathbb C}{\mathbb P}^3$ and the flag
complex manifold $F_3$. The sphere $S^4$ is conformally flat, so the
Atiyah-Hitchin-Singer almost complex structure on both twistor
spaces ${\mathcal Z}_{+}$ and ${\mathcal Z}_{-}$ of $S^4$ is
integrable. It coincides with the complex structure of ${\mathbb
C}{\mathbb P}^3$.  The manifold ${\mathbb C}{\mathbb P}^2$ with the
orientation induced by its complex structure is self-dual, but not
anti-self-dual. The  Atiyah-Hitchin-Singer almost complex structure
is integrable only on ${\mathcal Z}_{-}$ and it coincides on this
twistor space with the complex structure of $F_3$. We recall now how
the points of ${\mathbb C}{\mathbb P}^3$ and $F_3$ determine complex
structures on the tangent spaces of the corresponding base manifolds
compatible with the metric and $\pm$ the orientation.

In order to deal with the twistor space of $S^4$, we identify $S^4$
with  quaternionic projective line ${\mathbb H}{\mathbb P}^1$.
Writing quaternions as $z_1+z_2j$ with $z_1,z_2\in{\mathbb C}$, the
projection map $\pi:{\mathbb C}{\mathbb P}^3\to {\mathbb H}{\mathbb
P}^1$ is given in homogeneous coordinates by $[z_1,z_2,z_3,z_4]\to
[z_1+z_2j,z_3+z_4j]$.  We orient the space of quaternions by means
of the basis $(1,i,j,k)$. Consider the sphere $S^7$ as a submanifold
of ${\mathbb H}^2$. The group $Sp(1)$ of unit quaternions acts on
$S^7$ by left multiplication and ${\mathbb H}{\mathbb P}^1$ is the
quotient space of $S^7$ under this action of $Sp(1)$. Denote the
quotient map by $\rho:S^7\to{\mathbb H}{\mathbb P}^1$. Let
$\zeta=[z]\in{\mathbb C}{\mathbb P}^3$, where $z\in{\mathbb
C}^4={\mathbb H}^2$ and $||z||=1$. Then $\rho(z)=\pi([z])$. Moreover
$Ker\,\rho_{\ast\,z}=span\{iz,jz,kz\}\subset T_zS^7$. Let ${\mathcal
H}_z$ be the orthogonal complement of $span\{iz,jz,kz\}$ in the
tangent space $T_zS^7$. Then ${\mathcal H}_z$ is invariant under
multiplications by $i,j,k$ and $\rho_{\ast\,z}|{\mathcal H}_{z}$ is
an isomorphism onto $T_{\pi(\zeta)}{\mathbb H}{\mathbb P}^1$. Let
$I$ be the complex structure on the vector space ${\mathcal H}_{z}$
defined by multiplication by $i$. Then
$\widetilde{I}=\rho_{\ast\,z}\circ I\circ (\rho_{\ast\,z}|{\mathcal
H}_{z})^{-1}$ is a complex structure on $T_{\pi(\zeta)}{\mathbb
H}{\mathbb P}^1$ compatible with the metric and the orientation. If
we consider the space ${\mathbb H}$ with the opposite orientation,
then $\widetilde{I}$ is compatible with the metric and the opposite
orientation of $T_{\pi(\zeta)}{\mathbb H}{\mathbb P}^1$. The complex
structure $\widetilde{I}$ does not depend on the choice of a
representative $z$ of the point $\zeta\in{\mathbb C}{\mathbb P}^3$.
We refer to \cite[Sec. 5.12]{Will} for details.

\smallskip

Now, consider the  complex flag manifold $F=F_3$. Recall that its
points are pairs $(l,m)$ of a complex line $l$ and a complex plane
$m$ in ${\mathbb C}^3$ such that $l\subset m$.  In this setting, the
projection map $\pi:{\mathcal Z}_{\pm}=F\to {\mathbb C}{\mathbb
P}^2$ is $(l,m)\to l^{\perp}\cap m$, where $l^{\perp}$ is the
orthogonal complement of $l$ in ${\mathbb C}^3$ with respect to the
standard Hermitian metric of ${\mathbb C}^3$. It is convenient to
set $E_1=l$, $E_2=l^{\perp}\cap m$, $E_3=m^{\perp}$ so that to
identify the points of $F$ with the triples $(E_1,E_2,E_3)$ of
mutually orthogonal complex lines in ${\mathbb C}^3$ with
$\oplus_{i=1}^{3}E_i={\mathbb C}^3$. Then the projection map $\pi$
sends $\sigma=(E_1,E_2,E_3)$ to $E_2$. Its fibre is $\{E_1:~
E_1~\rm{a~ complex ~line~ in}~ E_2^{\perp}\}\cong {\mathbb
C}{\mathbb P}^1$. The tangent space of the flag manifold $F$ at
$\sigma$  is isomorphic to $Hom(E_1,E_2)\oplus Hom(E_1,E_3)\oplus
Hom(E_2,E_3)$ (see, for example, \cite{Lam}). The embedding of, say,
$Hom(E_1,E_3)$ is defined as follows. For $f\in Hom(E_1,E_3)$ and
$t\in{\mathbb R}$, let $\Gamma_{f}(t)= \{x+tf(x):~x\in E_1\}$ be the
graph of the map $tf$ in $E_1\oplus E_3$. Then
$c_f(t)=(\Gamma_f(t),E_2,\Gamma_f(t)^{\perp})$ is a smooth curve in
$F$ passing through $\sigma$, and the map $f\to \dot c_f(0)$ is an
embedding of $Hom(E_1,E_3)$ into $T_{\sigma}F$. Similarly for
$Hom(E_1,E_2)$ and $Hom(E_2,E_3)$. Clearly $\pi\circ c_f(t)\equiv
E_2$. Therefore $Ker\,\pi_{\ast\,\sigma}=Hom(E_1,E_3)$ and the
restriction of $\pi_{\ast\,\sigma}$ to $Hom(E_1,E_2)\oplus
Hom(E_2,E_3)$ is a vector space isomorphism onto
$T_{\pi(\sigma)}{\mathbb C}{\mathbb P}^2=Hom(E_2,E_1\oplus E_3)=
Hom(E_2,E_1)\oplus Hom(E_2,E_3)$. In particular, we see that the map
$\pi$ is neither holomorphic nor anti-holomorphic. The
multiplication by $i$ in both $Hom(E_1,E_2)$ and $Hom(E_2,E_3)$
defines a complex structure on the vector space $Hom(E_1,E_2)\oplus
Hom(E_2,E_3)$. Transferring this complex structure to
$T_{\pi(\sigma)}{\mathbb C}{\mathbb P}^2$ by means of the map
$\pi_{\ast\,\sigma}$, we obtain a complex structure compatible with
the metric and the opposite orientation of ${\mathbb C}{\mathbb
P}^2$. In order to obtain a complex structure compatible with the
metric and the standard orientation of  ${\mathbb C}{\mathbb P}^2$,
we transfer the complex structure on  $Hom(E_1,E_2)\oplus
Hom(E_2,E_3)$, which is multiplication by $-i$ on $Hom(E_1,E_2)$ and
by $i$ on $Hom(E_2,E_3)$.

\medskip

\centerline{{\bf Acknowledgements}}

 The author would like to thank the referee for his/her remarks.

\end{document}